\documentclass[11pt]{amsart}
\usepackage{amsmath,amsthm,amsfonts,latexsym,amssymb}
\usepackage{graphicx}
\usepackage{float}
\usepackage{amssymb}
\usepackage{amsmath}
\usepackage{latexsym}
\usepackage{indentfirst}
\usepackage{amsxtra}
\usepackage[utf8]{inputenc}
\usepackage{times}
\usepackage{tikz}
\usepackage{tikz-cd}
\usepackage{float}
\usepackage{import}
\usepackage{pdfpages}
\usepackage{epic}
\usepackage[dvipsnames]{xcolor}
\usepackage{comment}

\calclayout\evensidemargin .1in

\definecolor{brightpink}{rgb}{1.0, 0.0, 0.5}
\definecolor{fluorescentpink}{rgb}{1.0, 0.08, 0.58}
\definecolor{pembe}{rgb}{355, 0, 255}
\definecolor{gri}{rgb}{0.5, 0.5, 0.5}

\newtheorem{theorem}{Theorem}[section]
\newtheorem{lemma}[theorem]{Lemma}
\newtheorem{proposition}[theorem]{Proposition}
\newtheorem{corollary}[theorem]{Corollary}

\newtheorem{remark}[theorem]{Remark}

\title{LEGENDRIAN NON-SIMPLE WHITEHEAD DOUBLES OF THE TREFOIL}

\author{SALİHA KIVANÇ}
\address{Department of Mathematics, Hacettepe University, 06800 Beytepe-Ankara, \MakeUppercase{Türkİye}}
\email{salihakivanc@hacettepe.edu.tr}
\thanks{The author was partially supported by the Scientific and Technological Research Council of Turkey (TÜBİTAK) under the program BİDEB 2211-A National PhD Scholarship Programme and 2214-A  International Research Fellowship Programme for PhD Students.}

\begin{document}
	
\begin{abstract}
 Ozsváth and Stipsicz  showed that  some Eliashberg–Chekanov twist knots, which are Whitehead doubles of the  unknot, are not Legendrian simple. 
 We extend their result by considering some  Whitehead doubles of the trefoil: Using properties of knot Floer homology and the distinguished surgery triangle, we show that this family of knots is  Legendrian non-simple in the standard contact 3-sphere. 
\end{abstract}
	
\maketitle

 \section{Introduction}
A Legendrian knot is a smooth embedding of a circle  $S^{1}$ into a contact 3-manifold  $(Y,\xi)$ such that at every point on the knot, its tangent vector lies within the contact plane at that point \cite{MR2397738}. This  definition leads to distinct geometric and algebraic properties, prompting the development of specialized numerical invariants to classify them. Among the most fundamental numerical invariants are the Thurston-Bennequin invariant $tb$, which measures how the  contact structure twists around the knot, and the rotation number $rot$, which quantifies the winding of the knot within the contact planes. A key classification problem for Legendrian knots is whether they are Legendrian simple, that is, they have a  unique Legendrian isotopy class with the same invariants $tb$ and $rot$ within their knot type. \\

The study of Legendrian simple knots has a rich history in contact topology. Early work aimed at establishing the existence of Legendrian non-simple knots faced profound challenges. A significant breakthrough was achieved through Eliashberg’s foundational work \cite{MR1215964}, and independently by Chekanov \cite{MR1946550}. Subsequent contributions to this area can be found in \cite{MR1959579, MR3085098}.  More recent investigations, such as  \cite{MR3990225} have explored Legendrian simplicity in the context of knot Floer homology, thereby providing new perspectives on the classification of Legendrian simple knots and enriching the broader landscape of Legendrian knot theory.\\

In this work, we study the Legendrian simplicity of the Whitehead doubles of the trefoil knot. This topic has attracted considerable  interest in recent works such as \cite{MR3085098} and \cite{ MR2650809}, owing to its deep connections with knot Floer homology and numerical invariants of Legendrian knots. The main theorem of this paper is as follows.

\begin{theorem} \label{main}    
For all $n \geq 1$, the positively clasped $(-n)$-twisted Whitehead double of the right-handed trefoil is not Legendrian simple.
More precisely, there exist at least $\lfloor \frac{n+3}{2} \rfloor$ pairwise non-Legendrian isotopic realizations of such a knot in the standard contact $3-$sphere $(S^3,\xi_{st})$ with Thurston-Bennequin invariant $1$ and rotation number $0$.

\end{theorem}

\begin{figure}[ht!] 
    \def\svgwidth{.6\columnwidth}
    \centering
    %
\begingroup%
  \makeatletter%
  \providecommand\color[2][]{%
    \errmessage{(Inkscape) Color is used for the text in Inkscape, but the package 'color.sty' is not loaded}%
    \renewcommand\color[2][]{}%
  }%
  \providecommand\transparent[1]{%
    \errmessage{(Inkscape) Transparency is used (non-zero) for the text in Inkscape, but the package 'transparent.sty' is not loaded}%
    \renewcommand\transparent[1]{}%
  }%
  \providecommand\rotatebox[2]{#2}%
  \newcommand*\fsize{\dimexpr\f@size pt\relax}%
  \newcommand*\lineheight[1]{\fontsize{\fsize}{#1\fsize}\selectfont}%
  \ifx\svgwidth\undefined%
    \setlength{\unitlength}{708.66141732bp}%
    \ifx\svgscale\undefined%
      \relax%
    \else%
      \setlength{\unitlength}{\unitlength * \real{\svgscale}}%
    \fi%
  \else%
    \setlength{\unitlength}{\svgwidth}%
  \fi%
  \global\let\svgwidth\undefined%
  \global\let\svgscale\undefined%
  \makeatother%
  \begin{picture}(1,0.82799994)%
    \lineheight{1}%
    \setlength\tabcolsep{0pt}%
    \put(0,0){\includegraphics[width=\unitlength,page=1]{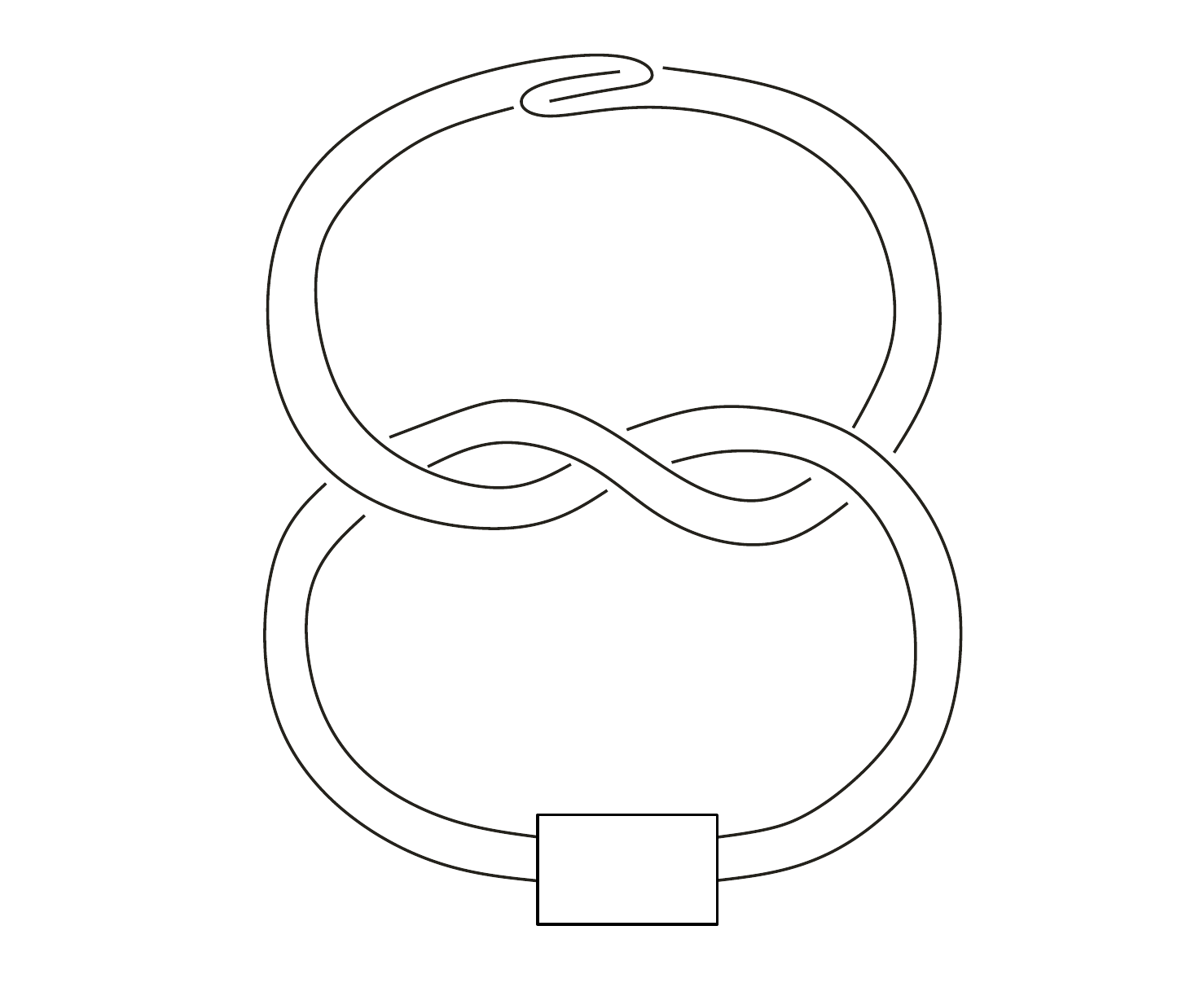}}%
    \put(0.46236893,0.09381){\color[rgb]{0,0,0}\makebox(0,0)[lt]{\lineheight{1.25}\smash{\begin{tabular}[t]{l}$-n-3$\end{tabular}}}}%
  \end{picture}%
\endgroup%

       \caption{The knot $Wh_{+}(S,-n)$, which is the positively clasped, $(-n)$-twisted Whitehead double of the right handed trefoil.
       Note that the box containing a $(-n-3)$ denotes  $n+3$ negative full twists. }
    \label{fig:st}
\end{figure} 

Our proof of Theorem 1.1 closely follows the methodology established in \cite{MR2650809}. The central idea is to distinguish various Legendrian realizations, denoted as $S(k,l)$, of the same underlying knot type $Wh_{+}(S,-n) \subset S^3$  (where $S$ is a trefoil), which also have the same classical invariants $tb$  and $rot$.  The core of our argument involves demonstrating that these $S(k,l)$ are indeed distinct, as evidenced by their Legendrian invariants $\widehat{\mathfrak{L}}(S(k,l))$ within the framework of knot Floer homology, specifically in its hat version $\widehat{HFK}(-S^3, Wh_{+}(S,-n))$. Knot Floer homology, introduced by Ozsváth and Szabó \cite{MR2065507} and independently by Rasmussen \cite{MR2704683}, is a refined invariant of knots and links embedded in $3-$manifolds, capturing fundamental topological information.  It can be obtained algebraically by defining a filtration on an appropriate chain complex of the Heegaard Floer homology $\widehat{HF}(Y)$ of the underlying 3-manifold. Geometrically, this construction involves adding an extra base point to the Heegaard diagram of the relevant 3-manifold. For each knot, the homology of the associated graded object provides a bigraded vector space endowed with Maslov and Alexander gradings.   Its foundation lies in symplectic topology drawing inspiration from Floer's work \cite{MR965228} on Lagrangian intersections and the study of holomorphic curves  in the symmetric products of Riemann surfaces. Knot Floer homology is also split  over relative $Spin^{c}$ structures  \cite{MR2113019}, and the invariants they encode provide topological information about the knot complement. \\

In our analysis, the hat version of Heegaard Floer homology, denoted as $\widehat{HFK}$, is a fundamental component. This version of the homology is characterized as a bigraded vector space. Its various properties,  including the $Spin^{c}$ decomposition, play a critical role in our investigations. Furthermore, we utilize two important invariants in Floer homology theories. First, the contact invariant denoted by $c(Y,\xi)$ is an element in $\widehat{HF}(-Y)$  associated with a given contact 3-manifold $(Y, \xi)$. Second, the Legendrian invariant $\widehat{\mathfrak{L}}$ is associated with a given Legendrian knot in a contact 3-manifold $(Y, \xi)$, and its gradings are computable by using $tb$ and $rot$ \cite{MR2650809}.   These invariants are  defined using a compatible open book decomposition \cite{MR2153455} and \cite{MR2557137}. The Legendrian invariant induces a homology class $\widehat{\mathfrak{L}}(Y, \xi, S) \in \widehat{HFK}(-Y, S)$, which is our main distinguishing tool. To achieve this, we identify the Legendrian knots $S(k,l)$ in the standard contact manifold $(S^3, \xi_{st})$ as a result of two specific contact surgeries and the application of the cancellation lemma in contact surgery theory.  The second surgery is performed on a contact manifold $(-Y', \xi(k,l))$ containing a Legendrian knot $S'(k,l)$.  Notably, all such  $S'(k,l)$ share the same underlying knot type $S' \subset Y'$. This surgery construction gives a map:

$$\widehat{F}_K: \widehat{HFK}(-Y', S') \longrightarrow \widehat{HFK}(-S^3,  Wh_{+}(S,-n))$$

This map $\widehat{F}_K$ sends the Legendrian invariant $\widehat{\mathfrak{L}}(S'(k,l))$ in the initial manifold to $\widehat{\mathfrak{L}}(S(k,l))$ in $S^3$.  Our approach builds on the surgery exact triangle in the hat version of knot Floer homology $\widehat{HFK}$. The key insight here is that the invariants $\widehat{\mathfrak{L}}(S'(k,l))$ are distinguished by being supported in different $\text{Spin}^{c}$ structures. $S^3$ has only one $\text{Spin}^{c}$ structure. However, using the properties of the surgery exact triangle in $\widehat{HFK}$, we can show that the map $\widehat{F}_K$ is injective within a specific Alexander grading.   This injectivity is the key to our argument, as it ensures that the initial distinctness of $\widehat{\mathfrak{L}}(S'(k,l))$ in different $\text{Spin}^{c}$, hence proving that the invariants $\widehat{\mathfrak{L}}(S(k,l))$ in $\widehat{HFK}(-S^3,  Wh_{+}(S,-n))$ remain distinct. \\

For clarity, we summarize the geometric meaning and relationships of the main variables used throughout this paper. The variable $n$ defines the family of twisted Whitehead doubles, where $n \in \mathbb{Z}$ indicates the number of twists in the Whitehead double knot construction. The specific family of positively clasped, twisted Whitehead doubles is referred to as $ Wh_{+}(S,-n)$. In the context of the Legendrian realizations $S(k,l)$ of the smooth knot $Wh_{+}(S,-n)$, the parameters $k$ and $l$ represent the number of additional positive crossings introduced on the left and right sides of a specific part of the Legendrian knot diagram, respectively, as detailed in Section \ref{sec:main}. Both $k$ and $l$ are non-negative odd integers with $k+l+2=2n+6$. The interplay between $n$, $k$, and $l$ dictates the precise Legendrian isotopy class being studied and allows for the construction of distinct Legendrian knots having the same classical invariants.\\

The paper is organized as follows: Section 2 provides the necessary preliminaries and notations. In Section 3, we present the proof of our main results.

\medskip \noindent {\em \textbf{Acknowledgments}.\/} I would like to express my sincere gratitude to my advisor András Stipsicz for his invaluable guidance and steady support throughout this research. I am also sincerely thankful to Marco Marengon,    Vikt\'oria F\"oldv\'ari for their insightful conversations and thoughtful contributions. 

\section{Preliminaries}
This section presents the foundations of knot Floer homology and the distinguished triangle of knots, outlining the essential definitions and constructions that allow for the calculation of Legendrian invariants of knots throughout the paper. \\

Let $K$ be a null-homologous knot in a $3-$manifold $Y$. Throughout this section, we use the notation $CFK(Y,K)$ and $HFK(Y,K)$ to denote either the hat or the minus version of knot Floer homology, unless otherwise specified.  The core idea of knot Floer homology involves the construction of a chain complex $CFK(Y, K)$ from geometric objects arising from a Heegaard diagram.  The simplest version $\widehat{HFK}(Y,K)$ is a finite dimensional bigraded vector space over $\mathbb{F}$, whereas $HFK^{-}(Y,K)$ is a finitely generated bigraded $\mathbb{F}[U]$-module. We will work with $\mathbb{F} = \mathbb{Z}/2\mathbb{Z}$ coefficients. Knot Floer homology is equipped with Maslov $d \in \mathbb{Q}$  grading and Alexander grading $s \in \mathbb{Z}$.

$$HFK(Y,K)=\bigoplus_{d \in \mathbb{Q}, \hspace{.1cm} s \in \mathbb{Z}} HFK_{d} (Y,K,s)$$

When considering a rational homology sphere $Y$, both the hat and minus versions of knot Floer homology can be decomposed over Maslov grading $M$ and $\text{Spin}^c$ structure $\mathfrak{t}$ as follows \cite{MR2113019}: 

$$HFK(Y,K)=\bigoplus_{d \in \mathbb{Q}, \hspace{.1cm} \mathfrak{t} \in Spin^{c}(Y,K)} HFK_{d} (Y,K,\mathfrak{t})$$

Here, $\text{Spin}^c(Y,K)$ denotes the set of \emph{relative} $Spin^{c}$ structures for $(Y,K)$, which is equivalent to $Spin^{c}$ structures on the manifold $Y_{0}(K)$ obtained by $0-$framed surgery along the knot $K$. For any chain complex element $x \in CFK(Y,K)$, let $\mathfrak{t}(x)$ denote the relative $\text{Spin}^c$ structure in $\text{Spin}^c(Y,K)$ to which $x$ contributes. For a chosen Seifert surface $F$ of the knot $K$ in $Y- \nu(K)$, let $\widehat{F}$ be the closed surface in $Y_0(K)$ obtained by capping off $F$. The Alexander grading of $x$ is then determined by the formula $A_{F}(x)=\frac{1}{2} \langle c_{1}(\mathfrak{t}(x)),[\widehat{F}] \rangle \in \mathbb{Q}$, where $c_{1}(\mathfrak{t})$ is the first Chern class of the $\text{Spin}^c$ structure $\mathfrak{t}$. For more detailed discussion of the relation between the elements of the chain complexes of knot Floer homologies and the relative $Spin^{c}$ structures, we refer the reader to \cite{MR2065507}. \\

For a more detailed discussion of the Legendrian invariant, we refer the reader to \cite{MR2650809} and \cite{MR2557137}. Moreover, the following theorem provides  the relations between the numerical invariants of a given Legendrian knot and the Maslov and Alexander gradings of its Legendrian invariant.

\begin{theorem} ({\cite[Theorem~1.6]{MR2650809}}) \label{thm:main}
   Let $L \subset (Y,\xi)$ be a null-homologous Legendrian knot in the contact 3–manifold $(Y,\xi)$. Suppose that $F$ is a Seifert surface for $L$. Then, the chain $\mathfrak{L}(L) \in CFK^{-}(-Y,L)$ is supported in Alexander grading
   $$2A_{F}(\mathfrak{L}(L))=tb(L)-rot_{F}(L)+1.$$
   If $c_{1}(\xi)$ is torsion, then the Maslov grading of $\mathfrak{L}(L)$ is determined by
   $$2A_{F}(\mathfrak{L}(L))-M(\mathfrak{L}(L))=d_{3}(\xi)$$
   where $d_{3}(\xi)$ is the 3–dimensional invariant of the 2–plane field underlying the contact structure $\xi$.
\end{theorem}
In particular, 3-dimensional invariant $d_{3}(\xi_{st})$ of  the standard contact structure on $S^3$ is equal to zero \cite{MR2114165}. 
\begin{remark}  \label{rk1}
Let $L$ be a Legendrian realization of a knot $K$ in the contact 3-manifold $(Y,\xi) $.   In computational settings, it is often more advantageous to use $\widehat{HFK}$. According to \cite{MR3838884}, there exists a map
 $$CFK^{-}(-Y,K) \xrightarrow{U=0}  \widehat{CFK}(-Y,K)$$
 which induces the following map on the homology  level
 $$HFK^{-}(-Y,K) \longrightarrow  \widehat{HFK}(-Y,K).$$
  This map preserves the Alexander grading and maps   the Legendrian invariant $\mathfrak{L}$ to  the element $\widehat{\mathfrak{L}}$.
\end{remark}

 \begin{remark} (\cite{MR2557137})  \label{rk} 
    Consider   the natural chain map $$CFK^{-}(-Y, K) \xrightarrow{U=1} \widehat{CF}(-Y)$$ which induces a map between the knot Floer and Heegaard Floer homologies $$HFK^{-}(-Y, K) \longrightarrow   \widehat{HF}(-Y).$$ 
    For a Legendrian realization $L \subset (Y,\xi)$ of $K$, it maps the Legendrian invariant $\mathfrak{L}(L)$  to the contact invariant $c(Y,\xi)$ of the contact structure $\xi$ on the 3-manifold $Y$.
\end{remark}

The following result is fundamental in contact surgery theory. It allows us, under specific conditions, to simplify contact structures by 'cancelling' certain surgeries.
\begin{proposition}  (Cancellation Lemma) \label{cancellation} \cite[Proposition~6.4.5]{MR2397738}
    Let $(Y^{'},\xi^{'})$ be the contact manifold obtained from $(Y,\xi)$ by contact $(-1)-$surgery along a Legendrian knot $K$ and contact $(+1)-$surgery along a Legendrian push-off $K^{'}$ of $K$. Then $(Y^{'},\xi^{'})$ is contactomorphic to $(Y,\xi)$.
\end{proposition}
\begin{figure}[ht!]
    \def\svgwidth{.6\columnwidth}
    \centering
    %
\begingroup%
  \makeatletter%
  \providecommand\color[2][]{%
    \errmessage{(Inkscape) Color is used for the text in Inkscape, but the package 'color.sty' is not loaded}%
    \renewcommand\color[2][]{}%
  }%
  \providecommand\transparent[1]{%
    \errmessage{(Inkscape) Transparency is used (non-zero) for the text in Inkscape, but the package 'transparent.sty' is not loaded}%
    \renewcommand\transparent[1]{}%
  }%
  \providecommand\rotatebox[2]{#2}%
  \newcommand*\fsize{\dimexpr\f@size pt\relax}%
  \newcommand*\lineheight[1]{\fontsize{\fsize}{#1\fsize}\selectfont}%
  \ifx\svgwidth\undefined%
    \setlength{\unitlength}{566.92913386bp}%
    \ifx\svgscale\undefined%
      \relax%
    \else%
      \setlength{\unitlength}{\unitlength * \real{\svgscale}}%
    \fi%
  \else%
    \setlength{\unitlength}{\svgwidth}%
  \fi%
  \global\let\svgwidth\undefined%
  \global\let\svgscale\undefined%
  \makeatother%
  \begin{picture}(1,0.44999996)%
    \lineheight{1}%
    \setlength\tabcolsep{0pt}%
    \put(-0.00030317,0.39982065){\color[rgb]{0,0,0}\makebox(0,0)[lt]{\lineheight{1.25}\smash{\begin{tabular}[t]{l}$\widehat{HFK}(Y^{'}, S^{'})$\end{tabular}}}}%
    \put(0.67327891,0.40170437){\color[rgb]{0,0,0}\makebox(0,0)[lt]{\lineheight{1.25}\smash{\begin{tabular}[t]{l}$\widehat{HFK}(Y^{''}, S^{''})$\end{tabular}}}}%
    \put(0.32671698,0.08335978){\color[rgb]{0,0,0}\makebox(0,0)[lt]{\lineheight{1.25}\smash{\begin{tabular}[t]{l}$\widehat{HFK}(Y^{'''}, S^{'''})$\end{tabular}}}}%
    \put(0.42099166,0.43748651){\color[rgb]{0,0,0}\makebox(0,0)[lt]{\lineheight{1.25}\smash{\begin{tabular}[t]{l}$\widehat{F}_{K}$\end{tabular}}}}%
    \put(0.64337008,0.20277226){\color[rgb]{0,0,0}\makebox(0,0)[lt]{\lineheight{1.25}\smash{\begin{tabular}[t]{l}$\widehat{F}_{K^{'}}$\end{tabular}}}}%
    \put(0.15186988,0.21526261){\color[rgb]{0,0,0}\makebox(0,0)[lt]{\lineheight{1.25}\smash{\begin{tabular}[t]{l}$\widehat{F}_{K^{''}}$\end{tabular}}}}%
    \put(0,0){\includegraphics[width=\unitlength,page=1]{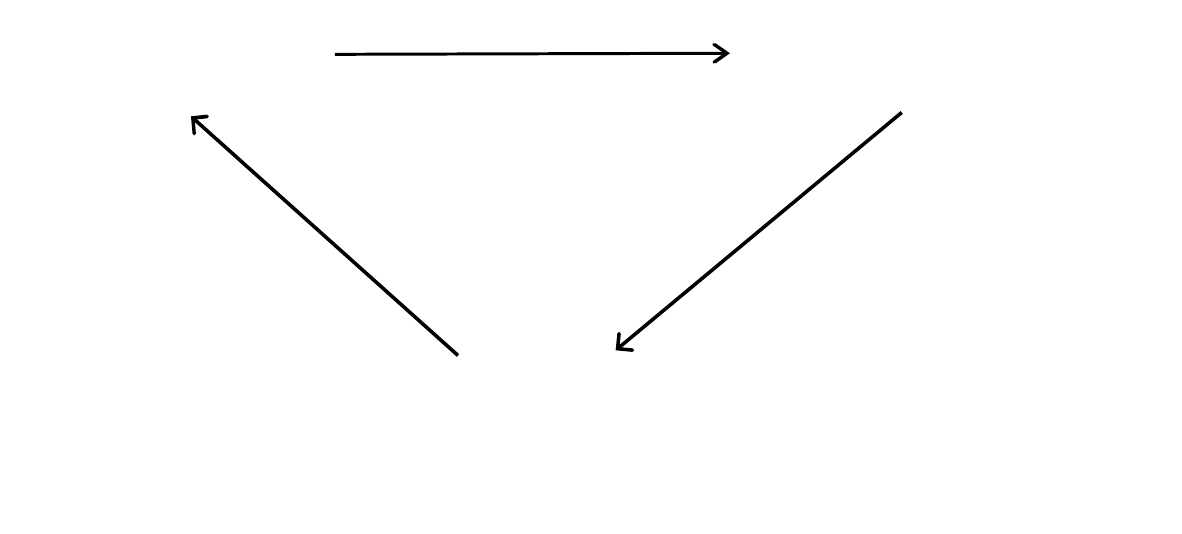}}%
  \end{picture}%
\endgroup%

    \caption{The distinguished surgery exact triangle on  knot Floer homology level}
    \label{triangle}
\end{figure}
To conclude, we recall the construction of the distinguished exact triangle involving knots and the induced maps on knot Floer homology. This exact triangle relates the knot Floer homologies of three manifolds obtained through a sequence of specific surgeries. Consider a link $L \cup \tilde{K} \subset S^{3}$. If surgery on the sublink $L$ with framing $\Lambda$ yields the manifold $Y'$, then performing surgery on $L \cup \tilde{K}$ with framing $\Lambda$ on $L$ and framings $\infty$, $0$, and $1$ on $\tilde{K}$ respectively. It yields the manifolds $Y'$, $Y''$, and $Y'''$. The image of $\tilde{K}$ in $Y'$ is precisely the knot we denote by $K$. \\

For clarity, we define the sequence of manifolds $Y'$, $Y''$, $Y'''$ and the relevant knots $K$, $K'$, $K''$ through a series of  surgeries on meridians. We begin with a $3-$manifold $Y'$ containing a knot $K$ with a chosen framing $f$. The manifold $Y''$ is obtained from $Y'$ by surgery along $K$ with framing $f$, denoted as $Y'' := Y'_{f}(K)$. The meridian of $K$, $\mu_{K} \subset Y'$, when given a framing $-1$, induces a framed knot $(K', f')$ in $Y''$. Likewise, the manifold $Y'''$ is obtained from $Y''$ by surgery along $K'$ with framing $f'$, denoted as $Y''' := Y''_{f'}(K')$. The meridian of $K'$, $\mu_{K'} \subset Y''$, with an induced framing $(-1)$, becomes a framed knot $(K'', f'')$ in $Y'''$. Lastly, surgery on $Y'''$ along $K''$ with framing $f''$ returns the original $3-$manifold $Y'$, i.e., $Y' = Y'''_{f''}(K'')$. Moreover, the meridian of $K''$, $\mu_{K''} \subset Y'''$, with an induced framing $(-1)$, induces the original framed knot $(K,f)$ in $Y'$. \\

The surgery maps associated with $K$, $K'$, and $K''$ (denoted $\widehat{F}_{K}$, $\widehat{F}_{K'}$, and $\widehat{F}_{K''}$, respectively) fit into an exact triangle on the Heegaard Floer homology level. This construction was further adapted in \cite{MR2650809} to a suitable version for knot Floer homology.   Figure \ref{triangle} illustrates this surgery exact triangle, showing the relevant knots $S'$, $S''$, $S'''$ within their respective surgered manifolds $Y'$, $Y''$, $Y'''$. Now, let's consider the conditions for grading preservation. If the surgered manifold $Y'$ is a rational homology sphere and the embedded knot $S'$ is null-homologous in $Y'-K$, then this implies the linking number between $S'$ and $K$ is trivial. Consequently, the induced map $\widehat{F}_{K}$ also preserves the Alexander grading \cite[Proposition~8.1]{MR2065507}. For a more detailed discussion of this construction and its properties, readers are referred to \cite[Chapter~14.3]{MR2114165} and \cite{MR2650809}.

\section{Legendrian Simplicity} \label{sec:main}
The aim of this section is to prove Theorem \ref{main}. Consider the Legendrian knot shown in Figure \ref{fig:legendrianskl}, where  $k$  and $l$ are odd positive integers.

\begin{figure}[ht!]
    \def\svgwidth{.9\columnwidth}
    \centering
    %
\begingroup%
  \makeatletter%
  \providecommand\color[2][]{%
    \errmessage{(Inkscape) Color is used for the text in Inkscape, but the package 'color.sty' is not loaded}%
    \renewcommand\color[2][]{}%
  }%
  \providecommand\transparent[1]{%
    \errmessage{(Inkscape) Transparency is used (non-zero) for the text in Inkscape, but the package 'transparent.sty' is not loaded}%
    \renewcommand\transparent[1]{}%
  }%
  \providecommand\rotatebox[2]{#2}%
  \newcommand*\fsize{\dimexpr\f@size pt\relax}%
  \newcommand*\lineheight[1]{\fontsize{\fsize}{#1\fsize}\selectfont}%
  \ifx\svgwidth\undefined%
    \setlength{\unitlength}{844.34378076bp}%
    \ifx\svgscale\undefined%
      \relax%
    \else%
      \setlength{\unitlength}{\unitlength * \real{\svgscale}}%
    \fi%
  \else%
    \setlength{\unitlength}{\svgwidth}%
  \fi%
  \global\let\svgwidth\undefined%
  \global\let\svgscale\undefined%
  \makeatother%
  \begin{picture}(1,0.69397782)%
    \lineheight{1}%
    \setlength\tabcolsep{0pt}%
    \put(0,0){\includegraphics[width=\unitlength,page=1]{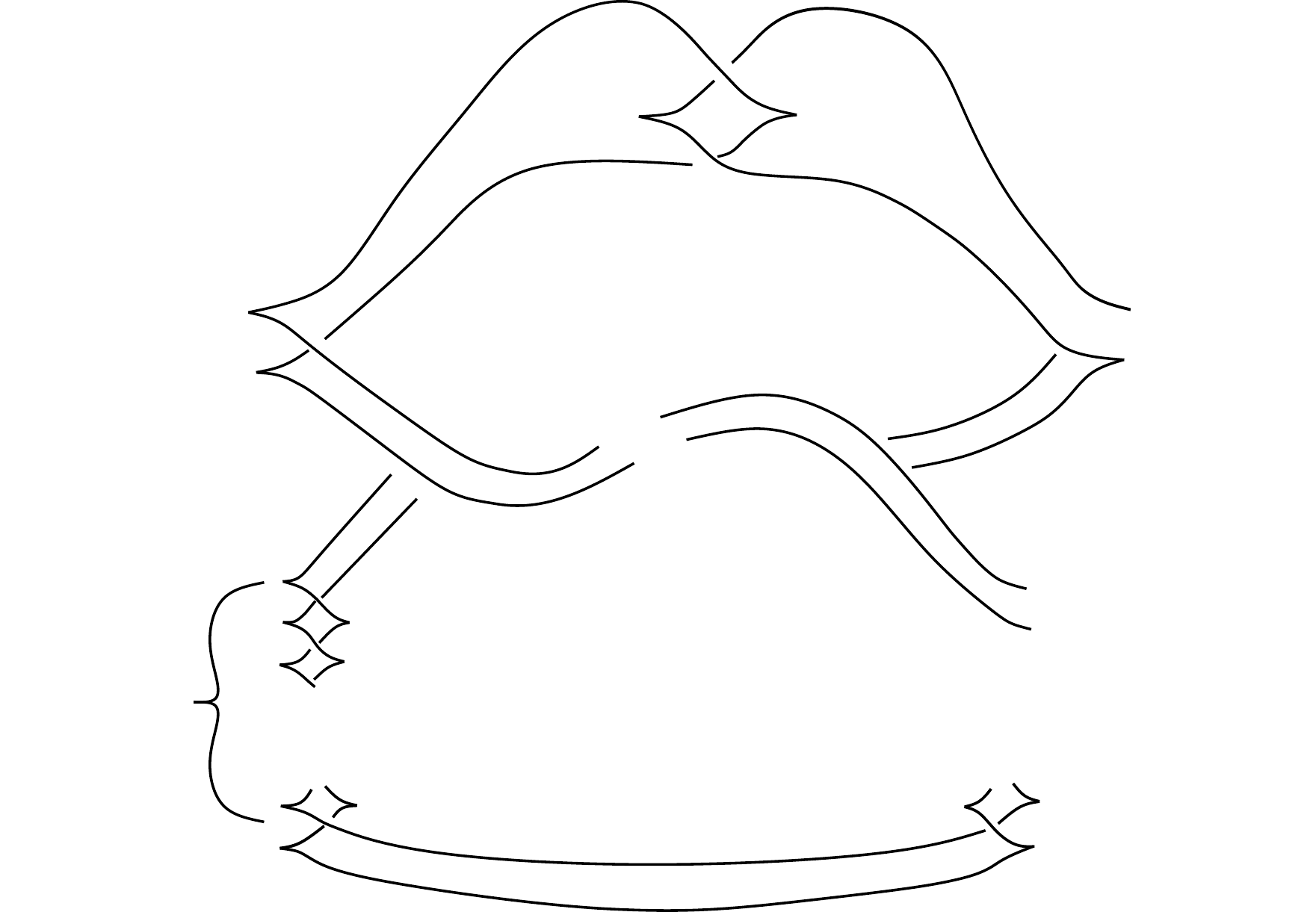}}%
    \put(0.00002748,0.15716945){\color[rgb]{0,0,0}\makebox(0,0)[lt]{\lineheight{1.25}\smash{\begin{tabular}[t]{l}$k$ crossings \end{tabular}}}}%
    \put(0,0){\includegraphics[width=\unitlength,page=2]{legendrianskl.pdf}}%
    \put(0.87912876,0.15950447){\color[rgb]{0,0,0}\makebox(0,0)[lt]{\lineheight{1.25}\smash{\begin{tabular}[t]{l}$l$ crossings \end{tabular}}}}%
    \put(0,0){\includegraphics[width=\unitlength,page=3]{legendrianskl.pdf}}%
  \end{picture}%
\endgroup%

    \caption{The Legendrian knot $S(k,l)$ in the standard $3-$sphere  $S^3$.}
    \label{fig:legendrianskl}
\end{figure}
\begin{proposition} \label{zero}
      The Legendrian knot  $S(k,l)$ in Figure \ref{fig:legendrianskl} has $tb(S(k,l))=1$, $rot(S(k,l))=0$ and it is smoothly isotopic to the knot $Wh_{+}(S,-n)$ described in Figure \ref{fig:st} with $2n+4=k+l \geq 2$.
\end{proposition}
\begin{proof}
In order to compute the numerical invariants, we use well-known formulas which can be found in 
\cite[Chapter~4]{MR2397738}. Specifically, the two formulas  are $tb(S(K,l))=writhe-\frac{c_{u}+c_{d}}{2}$ and $rot(S(K,l))=\frac{c_{d}-c_{u}}{2}$, where $c_{u}$ and $c_{d}$ denote the numbers of upward and downward cusps, respectively, in the Legendrian knot diagram. In Figure \ref{fig:legendrianskl}, one can easily see that the top 2 crossings have positive signs, and that there are two additional positive crossings in the upper part, one on the right and one on the left. The 12 crossings at the center of the diagram are grouped into three sets of 4, with their contributions to the writhe canceling out. At the bottom of the Figure \ref{fig:legendrianskl} on the left-hand side and  the right-hand side, there are $k$ and $l$ crossings, respectively. Each crossing is positive. In total, the writhe of the knot is $4+k+l$. \\

Now, let us compute the number of cusps in the knot diagram. Six of the cusps in the diagram are located in the upper part. Lastly, at the bottom of the diagram, there are $2k$  and $2l$ cusps on the left-hand side and right-hand side, respectively. Since strands appear in pairs with opposite directions, locally each pair consists of one upward cusp and one downward cusp. Consequently, there are $6+2k+2l$ cusps on the diagram. Therefore, we obtain the numerical invariants of $S(k,l)$ as  $tb(S(k,l))=1$ and $rot(S(k,l))=0$. Furthermore, consider the underlying smooth knot obtained by ignoring the cusps of $S(k,l)$. If one replaces all the crossings on the bottom left and right sides with a single box, this knot is smoothly isotopic to the knot shown in Figure \ref{fig:st} with $2n+6=k+l+2$. This isotopy establishes the relation between the numbers $k,l$ and $n$.
\end{proof}
\begin{proposition}
   The Legendrian invariant $\mathfrak{L}^{-}(S(k,l))$ is a nonzero element of the knot Floer homology group $HFK^{-}_{2}(S^{3}, m(Wh_{+}(S,-n)),1)$.
\end{proposition}

Note that for any knot $K$ in a $3-$manifold $Y$, one can consider the pair $(-Y, K)$, which is obtained by reversing the orientation of the ambient space $Y$ while keeping the knot $K$ as the same submanifold. When the ambient manifold is $Y = S^{3}$, there exists an orientation-preserving diffeomorphism between the pair $(-S^{3}, K)$ and $(S^{3}, m(K))$, where $m(K)$ denotes the mirror image of the knot $K$.
\begin{proof}
  The contact invariant of a Stein fillable contact 
  $3–$manifold  is non-zero  \cite[Theorem~1.5]{MR2153455}, see also \cite[Corollary~8.2.2]{MR2114165}. By applying \cite [Theorem~1.2]{MR2557137}, we conclude that the Legendrian invariant $\mathfrak{L}^{-}(S(k,l))$ is also non–zero. The Alexander and Maslov gradings of $\mathfrak{L}^{-}(S(k,l))$ are derived from the  formulas as specified  in Theorem \ref{thm:main}. 
\end{proof}

\begin{corollary} \label{minus}
The Legendrian invariant  $\widehat{\mathfrak{L}}(S(k,l))$ is a nonzero element of the knot Floer homology group $\widehat{HFK}_{2}(S^{3}, m(Wh_{+}(S,-n)),1)$.
\end{corollary}

\begin{proof}
The Seifert genus of the Whitehead double $Wh_{+}(S,-n)$  is equal to $1$ (see  \cite[Corollary~3.2]{MR1022995}). The Seifert genus of a knot is equal to the maximum Alexander grading of its knot Floer homologies $HFK^{-}$ and $\widehat{HFK}$ \cite[Theorem~1.2]{MR2023281}. 
Therefore, the maximum Alexander grading is $1$ for the knot $Wh_{+}(S,-n)$. This also holds for its mirror, $m(Wh_{+}(S,-n))$ since the Seifert genus is unchanged by mirroring. \\

Now, consider the Legendrian invariant $\widehat{\mathfrak{L}}$ associated with the  knot $Wh_{+}(S,-n)$. This invariant is an element of $\widehat{HFK}_{2}(S^{3}, m(Wh_{+}(S,-n)),1)$. By Theorem \ref{thm:main}, the Legendrian invariant $\mathfrak{L}$ resides in bidegree with Alexander grading $A=1$, Maslov grading $M=2$.
 Since the kernel of the specialization map $HFK^{-}(S^{3}, m(Wh_{+}(S,-n))) \longrightarrow \widehat{HFK}(S^{3}, m(Wh_{+}(S,-n)))$ is $U \cdot HFK^{-}(m(Wh_{+}(S,-n)))$, and $\mathfrak{L}$ resides in the maximal Alexander grading, it does not contain any term with $U$. Thus, $\mathfrak{L}$ is not in the kernel of this specialization map. Therefore, the image of the Legendrian invariant $\mathfrak{L}^{-}(S(k,l))$ under the specialization map remains nonzero in the Alexander grading $1$. The corollary follows directly from Remark \ref{rk1}.  
 \end{proof}

\color{black}
\begin{figure}[ht!]
    \def\svgwidth{\columnwidth}
    \centering
    %
\begingroup%
  \makeatletter%
  \providecommand\color[2][]{%
    \errmessage{(Inkscape) Color is used for the text in Inkscape, but the package 'color.sty' is not loaded}%
    \renewcommand\color[2][]{}%
  }%
  \providecommand\transparent[1]{%
    \errmessage{(Inkscape) Transparency is used (non-zero) for the text in Inkscape, but the package 'transparent.sty' is not loaded}%
    \renewcommand\transparent[1]{}%
  }%
  \providecommand\rotatebox[2]{#2}%
  \newcommand*\fsize{\dimexpr\f@size pt\relax}%
  \newcommand*\lineheight[1]{\fontsize{\fsize}{#1\fsize}\selectfont}%
  \ifx\svgwidth\undefined%
    \setlength{\unitlength}{816.45970142bp}%
    \ifx\svgscale\undefined%
      \relax%
    \else%
      \setlength{\unitlength}{\unitlength * \real{\svgscale}}%
    \fi%
  \else%
    \setlength{\unitlength}{\svgwidth}%
  \fi%
  \global\let\svgwidth\undefined%
  \global\let\svgscale\undefined%
  \makeatother%
  \begin{picture}(1,0.69177869)%
    \lineheight{1}%
    \setlength\tabcolsep{0pt}%
    \put(0,0){\includegraphics[width=\unitlength,page=1]{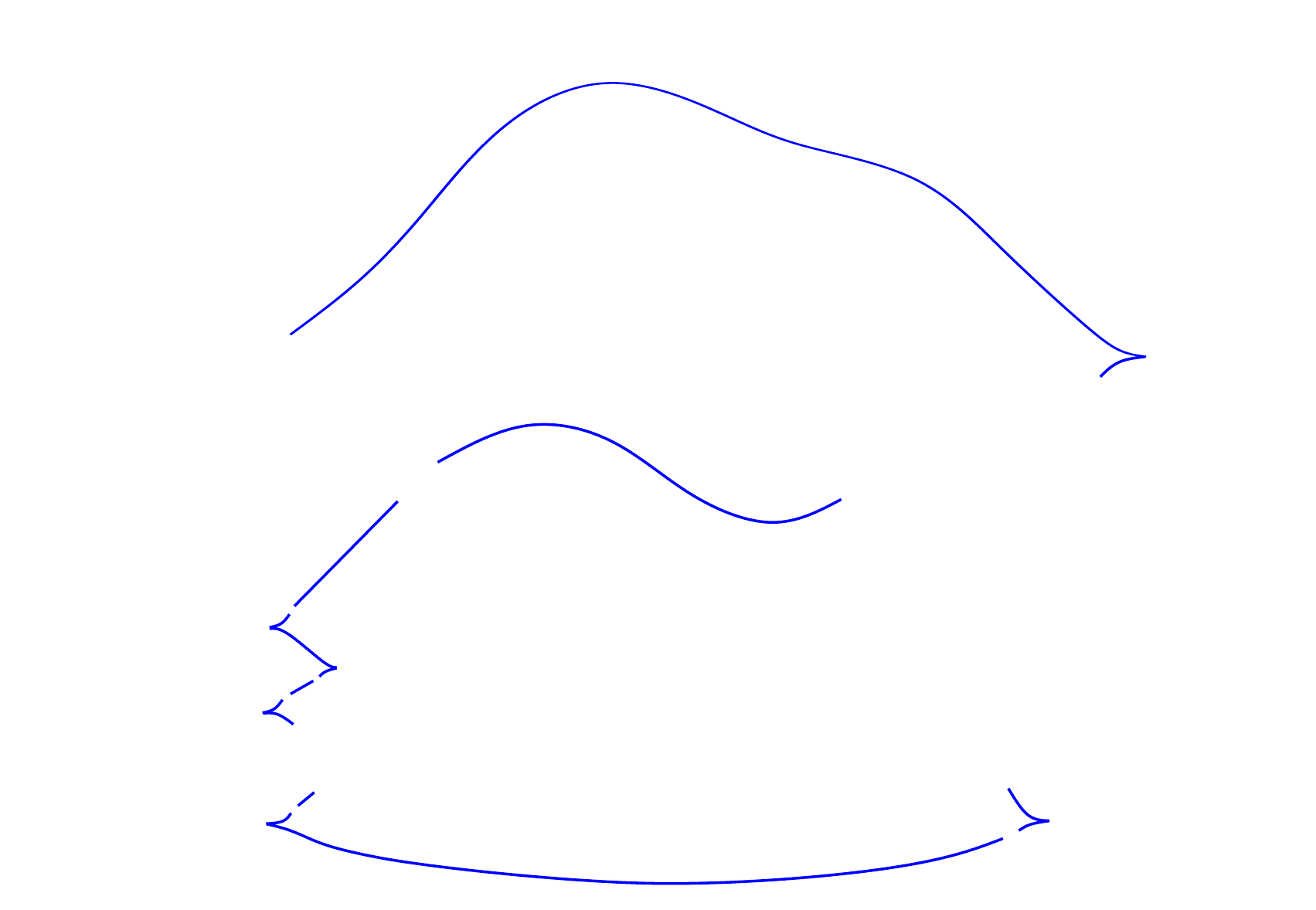}}%
    \put(0.87957267,0.1371868){\color[rgb]{0,0,0}\makebox(0,0)[lt]{\lineheight{1.25}\smash{\begin{tabular}[t]{l}$3l$ crossings\end{tabular}}}}%
    \put(-0.00003014,0.13966081){\color[rgb]{0,0,0}\makebox(0,0)[lt]{\lineheight{1.25}\smash{\begin{tabular}[t]{l}$3k$ crossings\end{tabular}}}}%
    \put(0.66791244,0.58090931){\color[rgb]{0,0,1}\makebox(0,0)[lt]{\lineheight{1.25}\smash{\begin{tabular}[t]{l}$-1$\end{tabular}}}}%
    \put(0,0){\includegraphics[width=\unitlength,page=2]{legendrianwith2h.pdf}}%
  \end{picture}%
\endgroup%

    \caption{The Legendrian knot $S(k,l)$ and the Legendrian trefoil $C(k,l)$ ({blue}) with  contact framing $(-1)$.}
    \label{fig:legendrianwith2h}
\end{figure}

\begin{figure}[ht!]
   \def\svgwidth{.9\columnwidth}
    \centering
    %
\begingroup%
  \makeatletter%
  \providecommand\color[2][]{%
    \errmessage{(Inkscape) Color is used for the text in Inkscape, but the package 'color.sty' is not loaded}%
    \renewcommand\color[2][]{}%
  }%
  \providecommand\transparent[1]{%
    \errmessage{(Inkscape) Transparency is used (non-zero) for the text in Inkscape, but the package 'transparent.sty' is not loaded}%
    \renewcommand\transparent[1]{}%
  }%
  \providecommand\rotatebox[2]{#2}%
  \newcommand*\fsize{\dimexpr\f@size pt\relax}%
  \newcommand*\lineheight[1]{\fontsize{\fsize}{#1\fsize}\selectfont}%
  \ifx\svgwidth\undefined%
    \setlength{\unitlength}{992.12581124bp}%
    \ifx\svgscale\undefined%
      \relax%
    \else%
      \setlength{\unitlength}{\unitlength * \real{\svgscale}}%
    \fi%
  \else%
    \setlength{\unitlength}{\svgwidth}%
  \fi%
  \global\let\svgwidth\undefined%
  \global\let\svgscale\undefined%
  \makeatother%
  \begin{picture}(1,0.59142863)%
    \lineheight{1}%
    \setlength\tabcolsep{0pt}%
    \put(0,0){\includegraphics[width=\unitlength,page=1]{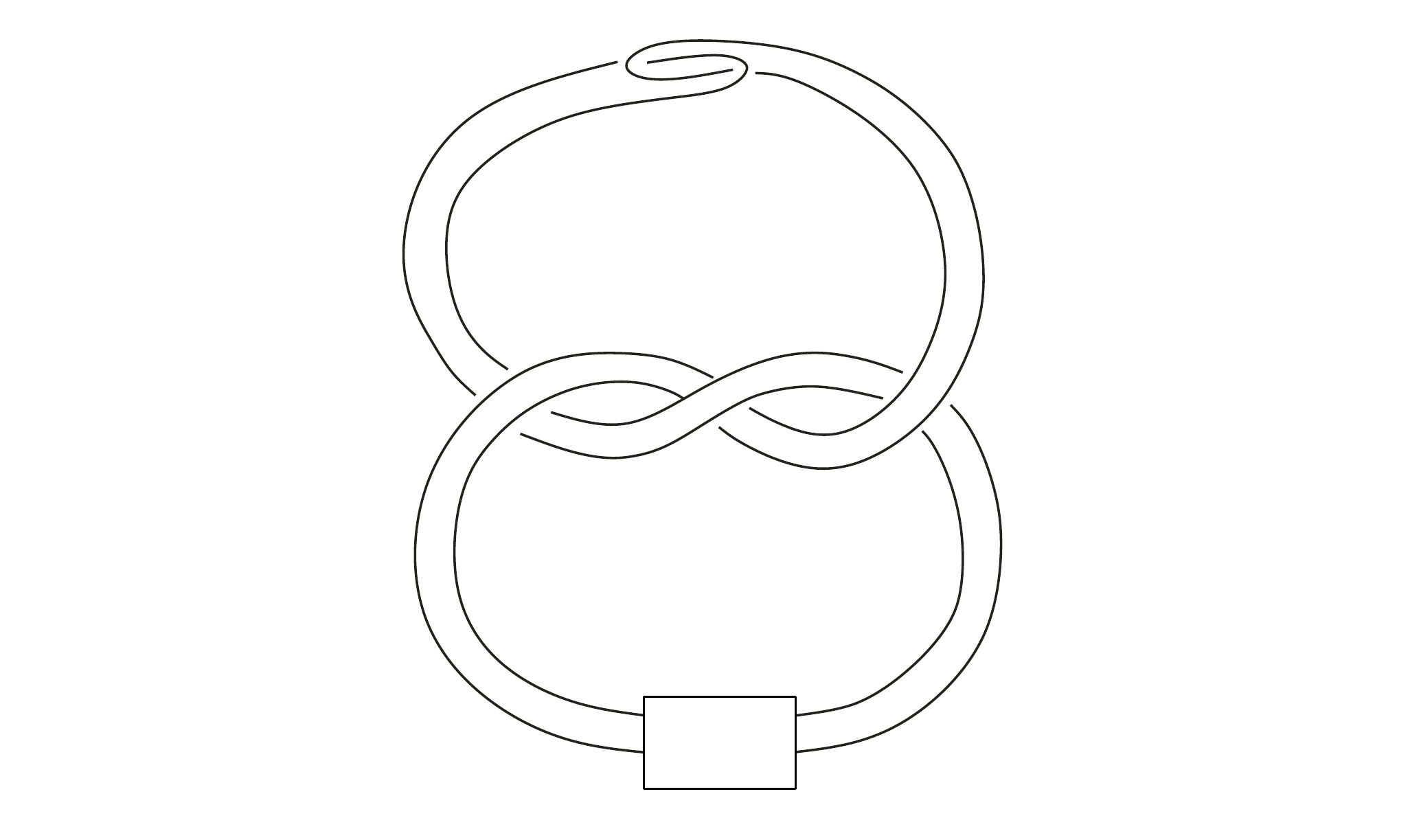}}%
    \put(0.46440529,0.06033187){\color[rgb]{0,0,0}\makebox(0,0)[lt]{\lineheight{1.25}\smash{\begin{tabular}[t]{l}$n+3$\end{tabular}}}}%
    \put(0,0){\includegraphics[width=\unitlength,page=2]{smthwith2handle.pdf}}%
    \put(0.65624327,0.54477655){\color[rgb]{0,0,1}\makebox(0,0)[lt]{\lineheight{1.25}\smash{\begin{tabular}[t]{l}$n+1$\end{tabular}}}}%
  \end{picture}%
\endgroup%

   \caption{ The mirror of the smooth knot $Wh_{+}(S,-n)$ and the left handed trefoil (blue) with Seifert framing $(n+1)$.}
    \label{fig:smthwith2handle}
\end{figure}

 Consider the $2–$component Legendrian link of  Figure \ref{fig:legendrianwith2h} which lies in $(S^{3}, \xi_{st})$.  Contact $(-1)-$surgery on $C(k,l)$ yields a contact manifold denoted $(-Y', \xi(k,l))$. The black-colored component of the link, namely $S(k,l)$, induces a Legendrian knot $S'(k,l)$ in $(-Y', \xi(k,l))$. Take the push-off of the Legendrian trefoil $C(k,l) \subset S^{3}$,  its image under this contact surgery gives  another natural Legendrian knot denoted by $C'(k,l)$ in $(-Y', \xi(k,l))$. Then we apply contact $(+1)-$surgery on $C'(k,l)$ in $(-Y', \xi(k,l))$ which according to the Proposition \ref{cancellation} cancels the first contact $(-1)-$surgery.  This  provides   the standard contact 
 $3–$sphere $(S^{3}, \xi_{st})$ with the Legendrian knot $S(k,l)$.  \\
  
We aim to construct the distinguished surgery triangle of knots induced by surgery along the trefoil knot (Figure \ref{exact}). The smooth description of the contact surgery operations explained earlier is as follows: The smooth diagram underlying Figure \ref{fig:legendrianwith2h} has a mirror given in Figure \ref{fig:smthwith2handle}. Let $C$ be the Legendrian push-off of the trefoil  in Figure \ref{fig:legendrianwith2h}, which serves as the companion knot in the Whitehead double construction. One can compute the Seifert framing $sf$ from the contact framing $cf$ by the formula $sf=tb-cf$ by \cite[Definition~3.5.4]{MR2397738}. When we do the contact $(-1)-$surgery, the Seifert framing of the smooth surgery knot is $(-(n+1))$. We now reverse the orientation of the ambient manifold $-Y'$. Then, the smooth manifold $Y'$ is obtained from $S^{3}$ by $(n+1)$ surgery on the mirror of $C$ i.e. $m(C)$.  Thus, the first term of the distinguished triangle  is represented by the smooth surgery diagram, a blue-colored knot depicted in the top left-hand side of Figure \ref{exact}.  We have the smooth knot $S'$  embedded  in the Seifert fibered space $Y^{'}$ as the black-colored knot  in the top left-hand side of Figure \ref{exact}.  $S'$ is given by the smooth knot type of $S'(k,l)$  upon changing orientation of the ambient manifold. Subsequently, the cancelling contact  $(+1)-$surgery on $C'(k,l)$  corresponds with the smooth surgery diagram at the top right-hand side of Figure  \ref{exact}. In the second term of the distinguished triangle, the manifold $Y^{''}$ has a smooth surgery diagram which consists of  the blue-colored knot with the Seifert framing $(n+1)$ and the red-colored  $0-$framed unknot. One can view the embedded black-colored knot in $S^3$ as the mirror of the smooth knot shown in Figure \ref{fig:st}, resulting from orientation reversal.  Here, the $0$-framed meridian can be canceled out with the blue-colored knot, yielding the standard diagram of $S^{3}$. Thus, $(Y'',S'')$ is exactly $(S^{3}, m(Wh_{+}(S,-n))) \cong (-S^{3}, Wh_{+}(S,-n))$. Lastly, the third term in the distinguished triangle is the manifold $Y'''$  whose smooth surgery diagram consists of the blue-colored knot with $n+1$ Seifert framing and an $1-$framed unknot in the bottom  of  Figure \ref{exact}. We have the embedded smooth knot in $Y'''$ which is black-colored knot $S'''$  induced under the surgery map.  Furthermore, the surgery  along the $0-$framed unknot $K$  between the first and the second term of the distinguished triangle induces  the following map: 
\begin{equation} 
   \widehat{F}_{K}: \widehat{HFK}(-Y^{'}, S') \longrightarrow \widehat{HFK}(-S^{3}, Wh_{+}(S,-n)) 
   \centering
    \label{iso}
\end{equation}

\begin{figure}[ht!]
    \def\svgwidth{\columnwidth}
    \centering
    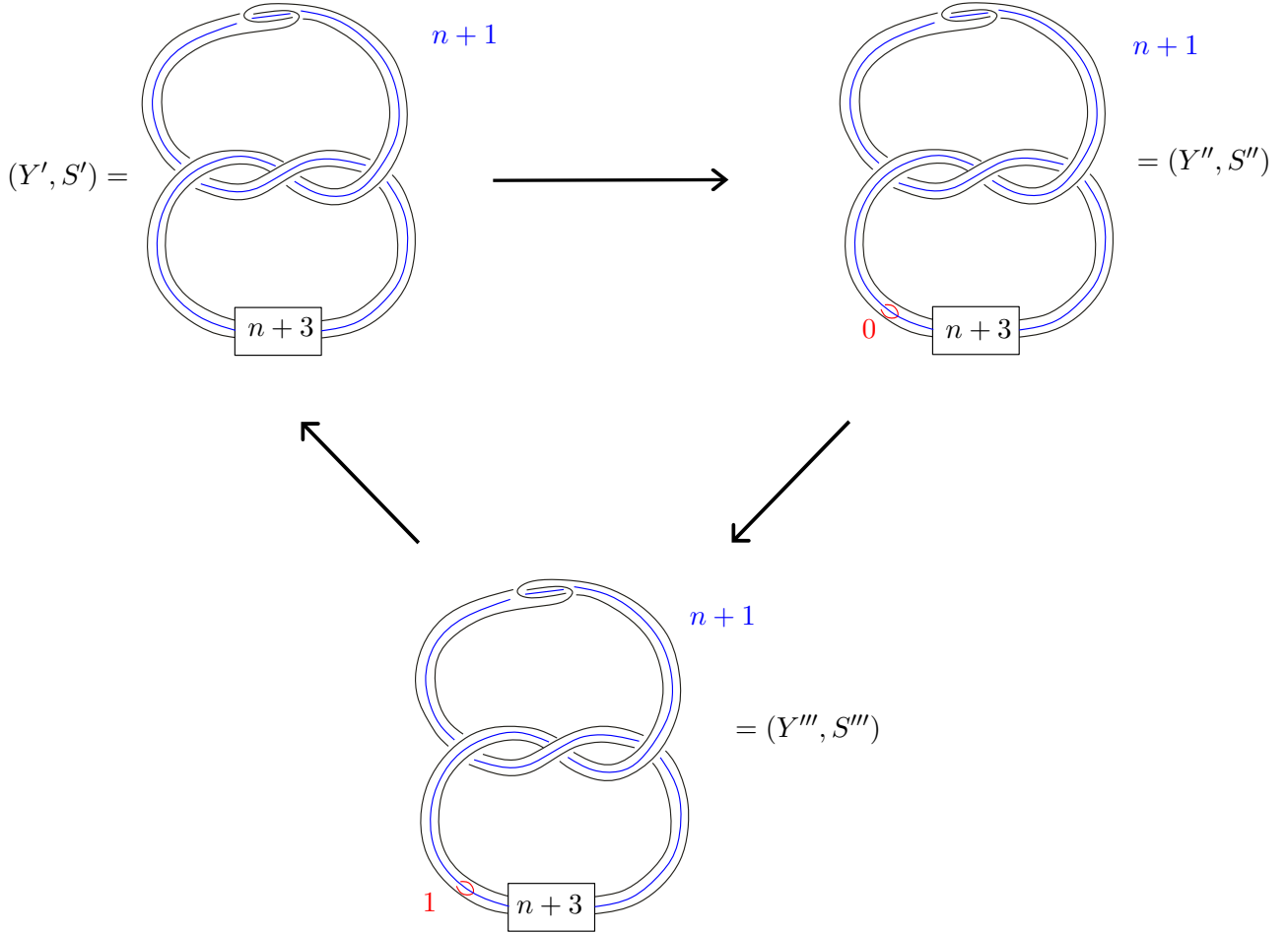

    \caption{The distinguished triangle of knots induced by the surgery along the trefoil.}
    \label{exact}
\end{figure}
\color{black}
\begin{lemma} \label{cal}
  The knot $S'''$ is an unknot in the Seifert fibered manifold  $Y^{'''}$  in Figure \ref{exact}.
\end{lemma} 

\begin{proof}
Initially, let us  blow-down   the unknot with framing $1$. The Seifert framing of the surgery trefoil changes to $n$. Subsequently, we take two push-offs of the 2-handle. We do two handle slides on the black-colored knot  along the 2-handle. In  Figure \ref{isotopy}(a), these two push-offs of the 2-handle are depicted as green and pink curves. These push-offs respect  the Seifert framing of the 2-handle and thus they link with the surgery knot $n$ times.  Then, we add bands to complete the handle slides as shown in Figure \ref{isotopy}(b). \\

After attaching the bands, one can simplify the diagram through a series of isotopies starting with the picture shown in Figure \ref{isotopy}(c). Through these isotopies, the diagram can be separated into an unknot and a blue-colored surgery knot with framing $n$, representing the surgery diagram of $Y^{'''}$.
\end{proof}

\begin{figure}[ht!]
    \def\svgwidth{\columnwidth}
    \centering
    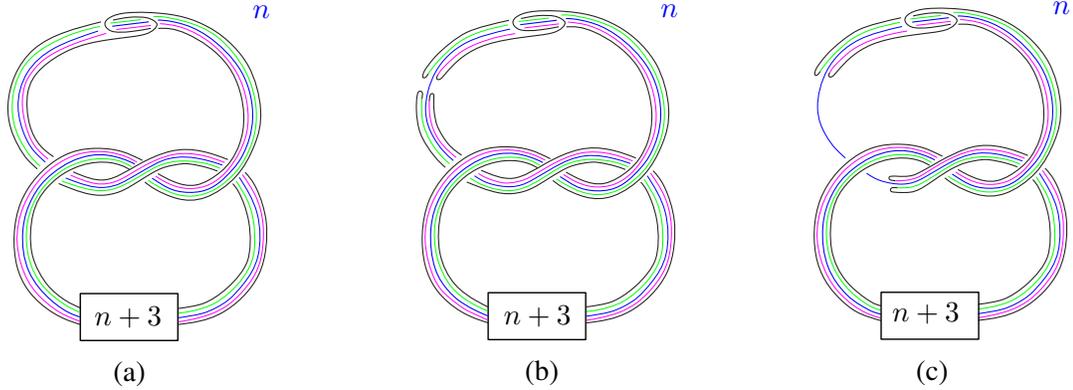

    \caption{(a)  The    green  and  pink  curves are two  push-offs of the  blue    colored surgery knot. (b) The band attachments for performing handle slides are shown.   (c) Starting the isotopy of the  knot is depicted.}
    \label{isotopy}
\end{figure}

\begin{corollary} \label{yok}
The knot Floer homology group   $\widehat{HFK}_{2}(Y^{'''},S''',1)$ vanishes.

\end{corollary}

 \begin{proof} 
By Lemma \ref{cal}, we know that the knot $S'''$ is an unknot in he manifold $Y'''$. Since the knot Floer homology of an unknot is supported only in Alexander grading $0$, its homology group in Alexander grading $1$ must vanish.
Consequently, the knot Floer homology  $\widehat{HFK}(Y^{'''},S''')$ vanishes at Alexander grading $1$. 
\end{proof}

Consider $Y^{'}$ as the oriented boundary of a smooth $4-$manifold $X$, built from a unique $0-$handle and a unique $2-$handle whose attaching sphere is the smooth knot $m(C) \subset S^{3}$ with  framing $n+1$. One can construct a closed surface $F$ from a Seifert surface   of the surgey knot in the $0-$handle  by capping it off with the core of the $2-$handle. 
 The cohomology class $[F]$ of the surface $F$ is a generator of $H^{2}(X;\mathbb{Z)} \cong \mathbb{Z}$.  The set of $Spin^{c}$ structures on $X$  forms an affine space over  $H^{2}(X;\mathbb{Z})$.  We label the $Spin^c$ structures on $X$ as $\mathfrak{s}^{X}_{j}$ for $j \in \mathbb{Z}$, uniquely determined by the condition:
$\langle c_{1}(\mathfrak{s}^{X}_{j}), [F] \rangle = n+1+2j$. Similarly, let $-X$ denote the manifold with reversed orientation, corresponding to surgery with framing $-n-1$. We label its $Spin^c$ structures $\mathfrak{s}^{-X}_{j}$ such that: $\langle c_{1}(\mathfrak{s}^{-X}_{j}), [F] \rangle = -n-1+2j$. Under the change of orientation, the labeling corresponds as $\mathfrak{s}^{X}_{j} \mapsto \mathfrak{s}^{-X}_{j+(n+1)}$. The restriction map $i^{\ast} : Spin^{c}(\pm X) \to Spin^{c}(Y')$ is surjective. Since there are $n+1$ distinct $Spin^c$ structures on the boundary $Y'$, denoted $\mathfrak{t}_{0}, \ldots, \mathfrak{t}_{n}$, the restriction is given by modulo arithmetic: $i^{\ast}(\mathfrak{s}^{\pm X}_{j}) = \mathfrak{t}_{[j]_{n+1}}$.\\

We now relate this labeling to the contact geometry. A Legendrian representative $C(k,l)$ of the knot $C$ induces an almost-complex structure $J_{k,l}$ on $-X$. This structure is supported by a specific $Spin^{c}$ structure $\mathfrak{s}^{-X}_{j}$ on $-X$ satisfying:
\[ \langle c_{1}(\mathfrak{s}^{-X}_{j}), [F] \rangle = \text{rot}(C(k,l)). \]
Combining this with the labeling definition for $-X$, we obtain the relation $\text{rot}(C(k, l)) = -n-1+2j$. Furthermore, the complex tangencies of $J_{k,l}$ on the boundary $\partial(-X) = -Y'$ induce the contact structure $\xi(k, l)$. Consequently, the contact structure $\xi(k, l)$ is supported by the boundary $Spin^{c}$ structure $\mathfrak{t}_{[j]_{n+1}}$. The restrictions of two $Spin^{c}$ structures $\mathfrak{s}_{j}$ and $\mathfrak{s}_{m}$ to the boundary $Y^{'}$  are equivalent if and only if $j \equiv m$ $(mod \hspace{.2cm} n+1)$ \cite[Chapter~6]{MR2114165}.

\begin{proposition} \label{long}
   Let $k$ and $l$ be odd integers satisfying $k \leq l$ and $k+l=2n+4$. The map $\widehat{F}_{K}$  given in Equation \ref{iso}  is injective on the subgroup of the knot Floer homology with Alexander grading $1$, and the Legendrian invariants $\widehat{\mathfrak{L}}(S'(k,l))$ are all non-zero and distinct.
\end{proposition}

\begin{proof}
The injectivity  of the map $\widehat{F}_{K}$ is a direct consequence of  Corollary \ref{yok} and the exactness of the distinguished triangle. 
 Let us consider two distinct Legendrian realizations of the trefoil in $(S^3, \xi_{st})$, denoted by $C(k_{1},l_{1})$ and $C(k_{2},l_{2})$ as depicted by the blue-colored knot in Figure \ref{fig:legendrianwith2h}. We choose these such that $k_{1} \neq k_{2}$ and $l_{1} \neq l_{2}$. The contact manifold  $(-Y^{'}, \xi(k,l))$  is obtained by performing contact $(-1)-$surgery on the   blue-colored Legendrian knot  $C(k,l) \subset (S^{3}, \xi_{st})$ in Figure \ref{fig:legendrianwith2h}, as discussed before Lemma \ref{cal}. The smooth surgery diagram of $-Y^{'}$ is depicted  at the top left of Figure \ref{exact}. The knot $S'(k,l)$ will denote a specific Legendrian realization of  $S'$ in  $(-Y', \xi(k,l))$, where $ \xi(k,l)$ is the contact structure induced by the surgery on  $C(k,l)$.  The choices of these distinct parameters $k$ and $l$, along with the condition $k+l=2n+4$, define distinct contact structures $\xi(k_{1},l_{1})$ and $\xi(k_{2},l_{2})$ on the manifold $-Y'$.\\

The rotation number of $C(k,l)$ can be computed using the combinatorial formula in \cite[Chapter~4]{MR2397738}. As depicted in Figure \ref{fig:legendrianwith2h}, if the top has one up cusp and one down cusp, and the bottom has $k$ down cusps and $l$ up cusps, then the rotation number of $C(k,l)$ is $\frac{k-l}{2}$. Alternatively, the rotation number of $C(k,l)$ is given by $rot(C(k,l))=\langle c_{1}(-X,J),[F] \rangle =n+1+2j$ for some $j\in \mathbb{Z}$. Since the Chern class of the almost-complex structure $J$ determines a $\text{Spin}^{c}$ structure $\mathfrak{s}_{j}$ on $-X$, the integer $j$ is  related to $\mathfrak{s}_{j}$. As $C(k_{1},l_{1})$ and $C(k_{2},l_{2})$ are chosen to have different rotation numbers (due to $k_1 \neq k_2$ and $l_1 \neq l_2$), their associated Chern classes must correspond to different $\text{Spin}^{c}$ structures on $-X$, denoted by $\mathfrak{s}_{j}$ and $\mathfrak{s}_{m}$, respectively. This yields the equalities $k_{1}-l_{1}=2n+2+4j$ and $k_{2}-l_{2}=2n+2+4m$. Then we determine their restricted $Spin^{c}$ structures on $-Y'$, given by $i^{\ast}(\mathfrak{s}_{j}) =\mathfrak{t}_{[j]_{n+1}} $ and $i^{\ast}(\mathfrak{s}_{m})=\mathfrak{t}_{[m]_{n+1}}$, where $j \equiv \frac{k_{1}-l_{1}-2n-2}{4} \pmod{n+1}$ and $ m \equiv \frac{k_{2}-l_{2}-2n-2}{4} \pmod{n+1}$.\\

In order to show that $\mathfrak{s}_{j}$ and $\mathfrak{s}_{m}$ induce different $Spin^{c}$ structures on $-Y'$, it suffices to prove that $j \not\equiv m \pmod{n+1}$. Using the relation $k+l=2n+4$ for the Legendrian knots $C(k,l)$, this condition simplifies to demonstrating that the values $\frac{k-1}{2}$ yield distinct residue classes modulo $n+1$.  By the relation, we know that there are  finitely many choices for $k$. If $n$ is even, then $k \in \{1, 3, \ldots n+3\}$.Then, we have $\frac{n+4}{2}$ many different choices for $k$. Otherwise $k \in \{1,  3, \ldots n+2\}$ and we have $\frac{n+3}{2}$ distinct choices for $k$. Indeed, for any two distinct odd  integers $k_1, k_2$ with the above conditions, we have $\frac{k_1-1}{2} \not\equiv \frac{k_2-1}{2} \pmod{n+1}$, which correspond to the different  $\text{Spin}^c$ structures. Therefore, the  different values of $k$ yield $\lfloor \frac{n+3}{2} \rfloor$ distinct $Spin^{c}$ structures on $-Y'$.\\

By the construction, the contact manifold $(-Y', \xi(k,l))$ obtained from contact $(-1)$-surgeries on Legendrian knot $C(k,l)$ are all Stein fillable. Consequently, their contact invariants $c(-Y', \xi(k,l))$ are all non-zero, as established by \cite[Corollary~8.2.2]{MR2114165} and \cite[Theorem~1.5]{MR2153455}. Furthermore, based on the preceding discussion regarding distinct restricted $\text{Spin}^c$ structures, the contact invariant $c(-Y', \xi(k,l))$ is supported by the $\text{Spin}^c$ structure $\mathfrak{t}_{[j]_{n+1}}$. Thus, for specific choices of $k$ (ranging from $1$ to $n+3$ if $n$ is even, or $1$ to $n+2$ if $n$ is odd), the $\text{Spin}^c$ structures supporting these contact invariants are all different. This demonstrates that contact $(-1)$-surgeries along distinct Legendrian realizations $C(k_1,l_1)$ and $C(k_2,l_2)$ induce different contact structures $\xi(k_1,l_1)$ and $\xi(k_2,l_2)$ on the same smooth manifold $-Y'$ \cite[Theorem~11.1.5]{MR2114165}. \\

Let $\widehat{\mathfrak{L}}(S'(k_1,l_1))$ and $\widehat{\mathfrak{L}}(S'(k_2,l_2))$ be the Legendrian invariants in $\widehat{HFK}(-Y',S')$ corresponding to the two Legendrian realizations $S'(k_{1},l_{1})$ and $S'(k_{2},l_{2})$ in $-Y'$, respectively.
The specialization map from Remark \ref{rk} relates the minus invariant $\mathfrak{L}^{-}(S'(k,l))$ to the contact invariant $c(-Y',\xi(k,l))$ when $U=1$. Specifically, $\mathfrak{L}^{-}(S'(k,l))$ specializes to $c(-Y',\xi(k,l))$. Since $c(-Y',\xi(k,l))$ is non-zero (as established in the preceding paragraph), it follows that $\mathfrak{L}^{-}(S'(k,l))$ is also non-zero. Furthermore, as $\mathfrak{L}^{-}(S'(k,l))$ is supported in a single $\text{Spin}^c$ structure of $Y'$, the specialization map implies that $\mathfrak{L}^{-}(S'(k,l))$ is supported in the same $\text{Spin}^c$ structure of $Y'$ as $c(-Y',\xi(k,l))$.\\

From the previous discussion, we know that $c(-Y',\xi(k_1,l_1))$ and $c(-Y',\xi(k_2,l_2))$ lie in different $\text{Spin}^c$ structures of $Y'$ (i.e., $\mathfrak{t}_{[j]_{n+1}}$ and $\mathfrak{t}_{[m]_{n+1}}$) when $k_1 \neq k_2$. By Corollary \ref{minus}, the invariants $\widehat{\mathfrak{L}}(S'(k,l))$ are non-zero. Since $\widehat{HFK}(-Y', S')$ decomposes over $\text{Spin}^c$ structures, the fact that these non-zero invariants belong to different $\text{Spin}^c$ structures implies that they are pairwise distinct. Thus, we conclude that all the Legendrian invariants $\widehat{\mathfrak{L}}(S'(k,l))$ are non-zero and pairwise distinct, thereby showing the existence of multiple non-Legendrian isotopic realizations of $S'$ in $-Y'$.

\end{proof}

Based on the finding of propositions above, we now proceed with the proof of the main theorem.

\begin{proof} [Proof of Theorem \ref{main}: ]
Let   us fix an integer $n \geq 1$ and consider the Whitehead double of the trefoil $Wh_{+}(S,-n)$ shown in Figure \ref{fig:st}.
We take two Legendrian realizations $S(k_{1},l_{1})$ and $S(k_{2},l_{2})$  of $Wh_{+}(S,-n)$ with  $k_{1} \neq k_{2}$, $l_{1} \neq l_{2}$. In order to prove that  $S(k_{1},l_{1})$ and $S(k_{2},l_{2})$  are not  Legendrian isotopic knots, we will show that their Legendrian invariants correspond to different elements in the knot Floer homology of $Wh_{+}(S,-n)$. Due to the naturality of Heegaard Floer homology \cite[Theorem~1.5]{MR4337438}, we have  an identification between $\widehat{HFK}(S^3, m(Wh_{+}(S,-n)))$ and $\widehat{HFK}(S^3, m(S(k,l)))$ for each Legendrian realization $S(k,l)$. \\

According to \cite[Theorem~2.3]{MR2059189} and Remark \ref{rk}, the map $\widehat{F}_{K}$ sends the Legendrian invariant in the first term of the distinguished triangle  to the Legendrian invariant of the Whitehead double of the trefoil. Since the  Legendrian invariants $\widehat{\mathfrak{L}}(S^{'}(k_{1},l_{1}))$ and $\widehat{\mathfrak{L}}(S^{'}(k_{2},l_{2}))$  are  different and  non-zero  in the subgroup $\widehat{HFK}_{2}(Y^{'},S^{'},1)$ in the first term of the distinguished exact triangle, the injectivity of    $\widehat{F}_{K}$  implies that the Legendrian invariants are distinct in $\widehat{HFK}_{2}(S^{3},m(Wh_{+}(S,-n)),1)$. Consequently, the knots  $S(k_{1},l_{1})$ and $S(k_{2},l_{2})$ define distinct, non-zero elements in the knot Floer homology of the Whitehead double of the trefoil, proving they are non-Legendrian isotopic.\\

\color{black}
In Figure \ref{fig:st}, the box represents full twists denoted by  $(-n-3)$  with the integer $n \geq 1$.  Our goal is to determine the number of distinct unordered pairs $(k, l)$ that correspond to the Legendrian realizations in Figure \ref{fig:legendrianskl}. This value is determined by the equation $2n+6=k+l+2$, where $k$ and $l$ are positive odd integers, $k \leq l$. The calculation depends on the parity of $n$. If $n$ is an odd integer, the pairs of integers range from $(1, 2n+3)$ to $(n+2, n+2)$. If $n$ is an even integer, the pairs range from $(1, 2n+3)$ to $(n+1, n+3)$.  In both cases, the number of  unordered pairs is  $\lfloor \frac{2n+6}{4} \rfloor$. These choices provide a lower bound for the number of different Legendrian realizations of $Wh_{+}(S,-n)$.  Thus, for all $n\geq 1$,  there are at least $\lfloor \frac{n+3}{2} \rfloor$ non-Legendrian isotopic realizations of $Wh_{+}(S,-n)$ in the standard contact 3-sphere $(S^3,\xi_{st})$  with Thurston-Bennequin invariant $1$ and rotation number $0$.
\end{proof}

\end{document}